\documentclass[letterpaper, 10pt, conference]{ieeeconf}
\IEEEoverridecommandlockouts 
\usepackage{amsmath,amssymb,url}
\usepackage{graphicx}
\usepackage{color}
\usepackage{cleveref}
\usepackage{booktabs}
\usepackage{tikz}
\usetikzlibrary{arrows.meta,calc}
\usepackage{subcaption}

\renewcommand{\H}{\mathsf{H}}
\newcommand{\X}{\mathsf{X}}

\newcommand{\G}{\ensuremath{\mathsf{G}}}

\renewcommand{\Re}{\ensuremath{\mathbb{R}}}

\newcommand{\pair}[1]{\ensuremath{\left\langle #1 \right\rangle}}

\newcommand{\Q}{\ensuremath{\mathsf{Q}}}

\newtheorem{theorem}{Theorem}

\title{\LARGE \bf Uncertainty Propagation in Stochastic Hybrid Systems with Dimension-Varying Resets}

\author{%
	Tejaswi K. C.\thanks{Tejaswi K. C., Mechanical and Aerospace Engineering, George Washington University, Washington, DC 20052. {\tt kctejaswi999@gmail.com}} and %
    Taeyoung Lee \thanks{Taeyoung Lee, Mechanical and Aerospace Engineering, George Washington University, Washington, DC 20052. {\tt tylee@gwu.edu}}%
    \thanks{The research was partially supported by AFOSR under the award MURI FA9550-23-1-0400.}
}

\newtheorem{definition}{Definition}

\graphicspath{{./figs/}}

\begin{document}
%\raggedright

\allowdisplaybreaks

\maketitle \thispagestyle{empty} \pagestyle{empty}

\begin{abstract}
	This paper studies probability density evolution for stochastic hybrid systems with reset maps that change the dimension of the continuous state across modes.
	Existing Frobenius--Perron formulations typically represent reset-induced probability transfer through boundary conditions, which is insufficient when resets map guard sets into the interior or onto lower-dimensional subsets of another mode.
	We develop a weak-form formulation in which reset-induced transfer is represented by the pushforward of probability flux across the guard, yielding a unified description for such systems.
	The proposed framework naturally captures both cases: when the reset decreases dimension, the transferred probability appears as an interior source density, whereas when the reset increases dimension, it generally appears as a singular source supported on a lower-dimensional subset.
	The approach is illustrated using a stochastic hybrid model in which two particles merge into one and later split back into two, demonstrating how dimension-changing resets lead to source terms beyond classical boundary-condition-based formulations.
\end{abstract}

\section{Introduction}

Hybrid systems integrate continuous-time evolution with discrete transitions, providing a unified framework for modeling complex systems that exhibit continuous and event-driven dynamics.
To incorporate uncertainty, stochastic hybrid system formulations have been developed.
The General Stochastic Hybrid System (GSHS) framework~\cite{bujorianu2006toward} models continuous evolution via mode-dependent stochastic differential equations, while discrete transitions occur either upon reaching guard sets or randomly according to a Poisson process.
Post-jump states are described by a stochastic kernel.
This framework has been applied in a wide range of domains~\cite{blom2009rare,pakdaman2010fluid,hespanha2004stochastic}.

For hybrid systems, uncertainty propagation is significantly more challenging due to the interplay between continuous dynamics and discrete events.
For deterministic hybrid systems, the Frobenius--Perron operator has been used to describe the evolution of densities~\cite{oprea2024study}.
For GSHS with spontaneous jumps, spectral methods have been developed to solve the associated hybrid Fokker--Planck equations~\cite{wanleea19,wanleesjads22}, with applications to Bayesian estimation.
More recently, it has been shown that different stochastic hybrid system formulations arise depending on how randomness is introduced in the continuous evolution and discrete transitions, leading to distinct dynamical behaviors~\cite{CClaLeePISNCS25}.

Despite these advances, existing approaches to uncertainty propagation in stochastic hybrid systems face a fundamental limitation.
Most existing formulations implicitly rely on transitions between modes whose continuous state spaces share the same dimension, or can be embedded into a common-dimensional representation.
Under this setting, probability transfer can be expressed through density mappings or boundary conditions within a standard Fokker--Planck framework.

However, many hybrid systems exhibit transitions between modes of different dimensions.
In such cases, probability mass is transferred across spaces of different dimensions, and the resulting source term may no longer be representable as a regular density.
Instead, it may become singular or concentrated on lower-dimensional sets, depending on the geometry of the reset map.
Consequently, existing density-based formulations do not provide a unified and geometrically consistent description of probability flow across modes.

In this paper, we develop a unified formulation of the hybrid Fokker--Planck equation based on a flux-driven, measure-theoretic perspective.
The key idea is to characterize the transfer of probability mass between modes through the pushforward of boundary flux induced by the probability current.
This leads to a general expression in which the source term in each mode is obtained by mapping boundary flux through the reset mechanism.
The resulting formulation naturally accommodates both smooth density contributions and singular components, depending on the relative dimensions of the interacting spaces.

Building on this formulation, we derive a consistent numerical scheme that preserves the flux-based structure of the hybrid dynamics.
In particular, the proposed finite-volume discretization directly approximates the probability current and its normal component at interfaces, ensuring that probability mass is transferred between modes in a conservative and physically consistent manner.

The proposed framework is illustrated through a representative example in which two particles evolve on a line and collapse into a single particle upon reaching a guard set, which may split into two parts again.
This example highlights how boundary flux in one mode induces either a smooth density or a singular source term in another mode, depending on the geometry of the reset map.

In summary, this paper develops a hybrid Fokker--Planck formulation for uncertainty propagation that consistently captures probability transport across modes with potentially different state dimensions, together with a conservative numerical scheme that preserves the underlying flux structure.

\section{Stochastic Hybrid Systems}

In this section, we formulate stochastic hybrid systems whose continuous state space may vary in dimension across discrete modes.
The presentation is restricted to Euclidean spaces for clarity.

\subsection{Stochastic Hybrid Systems}

Let $\Q=\{1,2,\ldots,N_Q\}$ be a finite set of discrete modes.
For each $q\in\Q$, the continuous state evolves in a domain
\begin{align}
    \X_q \subset \Re^{n_q},
\end{align}
where the dimension $n_q$ may depend on the mode.
The hybrid state space is defined as the disjoint union
\begin{align}
    \H = \bigsqcup_{q\in\Q} \X_q,
\end{align}
and a hybrid state is denoted by $(x,q)$ with $x\in\X_q$.

In each mode $q\in\Q$, the continuous state evolves according to the stochastic differential equation given by
\begin{align}
    dx = X_q(x)\,dt + \sigma_q\, dW_q,
    \label{eqn:SDE_mode}
\end{align}
where $X_q:\X_q\to\Re^{n_q}$ is a smooth vector field, $\sigma_q>0$ is a constant diffusion coefficient, and $W_q(t)$ is a standard Wiener process in $\Re^{n_q}$.

Each mode $q\in\Q$ is equipped with a guard set
\[
	\G_q \subset \X_q,
\]
which is a measurable subset of the boundary at which discrete transitions are triggered.
While guards are often codimension-one subsets of the boundary, the framework also allows lower-dimensional or degenerate cases, depending on the geometry of the reset mechanism.

Specifically, when the hybrid state $(x,q)$ reaches the guard $\G_q$, a discrete transition occurs instantaneously.
The post-transition state is determined by a reset map
\begin{align}
    \Phi: \bigsqcup_{q\in\Q} \G_q \to \H,
\end{align}
such that the hybrid state immediately after the jump is
\begin{align}
    (x^+, q^+) = \Phi(x, q).
\end{align}
for any $(x,q)\in \G_q$.
It is considered that the reset map is measurable.

\subsection{Dimension-Varying Reset Maps}

A key feature of this formulation is that the reset map may change the dimension of the continuous state.
More explicitly, for $(x,q)\in\G_q$, the dimension of the continuous state after the jump $n_{q^+}$ may not be identical to $n_q$.
This includes several important cases:
\begin{itemize}
    \item \textbf{Dimension reduction:} the state is mapped to a lower-dimensional space (e.g., merging or projection),
    \item \textbf{Dimension expansion:} the state is lifted to a higher-dimensional space (e.g., splitting or reconstruction),
    \item \textbf{Dimension-preserving reset:} the state remains in a space of the same dimension.
\end{itemize}

Next, for each mode $q\in\Q$, we assume that the boundary of $\X_q$ is decomposed as
\begin{align}
    \partial \X_q = \G_q \cup \Gamma_q, \qquad \G_q \cap \Gamma_q = \emptyset,
\end{align}
where $\G_q$ is the guard set that triggers discrete transitions, and $\Gamma_q$ is a reflecting boundary.
On $\Gamma_q$, the stochastic dynamics are constrained so that there is no probability flux across the boundary, or more precisely, the normal component of the probability current vanishes, ensuring that the continuous evolution remains within $\X_q$ between guard events.

On the guard $\G_q$, the process undergoes an instantaneous reset and does not persist on the guard.
Accordingly, the probability density $p: \Re\times \H\rightarrow\Re$ of the hybrid state has zero trace on the guard, or equivalently,
\begin{align}
    p(t,x,q) = 0, \qquad x\in\G_q,
\end{align}
and probability mass reaching $\G_q$ is transferred to other modes via the reset map $\Phi$.
\subsection{Hybrid Evolution}

The evolution of the stochastic hybrid system consists of alternating phases of continuous stochastic flow and discrete transitions:
\begin{itemize}
    \item Between guard events, the state evolves continuously according to \eqref{eqn:SDE_mode} within the current mode.
    \item When the state reaches a guard set $\G_q$, it is instantaneously reset by $\Phi$ and the mode may change.
\end{itemize}

This framework captures stochastic hybrid systems in which continuous diffusion processes interact with deterministic jump mechanisms, including transitions across state spaces of different dimensions.

\section{Stochastic Koopman Operator and Generator}

In this section, we formulate the stochastic Koopman operator for hybrid systems with dimension-varying reset maps and derive its infinitesimal generator.
The stochastic Koopman operator describes the evolution of observables along trajectories of a stochastic hybrid system, mapping each function to its expected future value.
Its infinitesimal generator characterizes the instantaneous rate of change of this expectation, combining the effects of continuous stochastic dynamics and discrete transitions.

\subsection{Koopman Operator}

\begin{definition}
Let $z_t = (x_t, q_t)$ denote the hybrid stochastic process evolving on $\H$.
The stochastic Koopman operator $\mathcal{K}_t$ is defined by
\begin{align}
    \mathcal{K}_t f(x,q) = \mathbb{E}\big[f(x_t, q_t)\mid (x_0,q_0)=(x,q)\big], \end{align}
for any measurable function $f:\H\to\Re$.
\end{definition}
Here $\mathcal{K}_t f(x,q)$ represents the expected value of the observable $f$ along trajectories of the hybrid system at $t$ when initialized with $(x,q)$.

\subsection{Infinitesimal Generator}

\begin{theorem}
Consider the stochastic hybrid system defined in Section II.
Let $f:\H\to\Re$ be a measurable function, and define
\[
	u(t,x,q) = \mathcal{K}_t f(x,q).
\]
Then, for each mode $q\in\Q$, the function $u$ satisfies
\begin{align}
    \frac{\partial u(t,x,q)}{\partial t} &= \mathcal{A}_q u(t,x,q), \qquad x \in \X_q \setminus \G_q, \label{eqn:koopman_pde}
\end{align}
with initial condition $u(0,x,q)=f(x,q)$, where the generator $\mathcal{A}_q$ is given by
\begin{align}
    \mathcal{A}_q u(x,q) = X_q(x)\cdot \nabla u(x,q) + \frac{\sigma_q^2}{2}\,\Delta u(x,q). \label{eqn:koopman_generator}
\end{align}

Moreover, on the guard $\G_q$, the Koopman observable satisfies the compatibility condition, given by
\begin{align}
    u(t,x,q) = u\big(t,\Phi(x,q)\big), \qquad x\in\G_q.
    \label{eqn:koopman_bc}
\end{align}
\end{theorem}

\begin{proof}
Fix $(x,q)\in\X_q \setminus \G_q$.
Between discrete transitions, the process evolves according to the stochastic differential equation \eqref{eqn:SDE_mode}.
Applying the standard generator of diffusion processes in $\Re^{n_q}$ yields
\begin{align}
    \mathcal{A}_q u = X_q\cdot \nabla u + \frac{\sigma_q^2}{2}\Delta u,
\end{align}
which gives \eqref{eqn:koopman_generator}.
The evolution equation \eqref{eqn:koopman_pde} follows from the definition of the infinitesimal generator of the stochastic Koopman operator~\cite{oksendal2003stochastic}.

Next, consider $x\in\G_q$.
By definition of the hybrid system, when the process reaches $(x,q)$, it is instantaneously reset to $\Phi(x,q)$.
Therefore, for any $t>0$,
\begin{align}
	u(t,x,q) &= \mathbb{E}\big[f(z_t)\mid z_0=(x,q)\big] \nonumber \\
			 &= \mathbb{E}\big[f(z_t)\mid z_0=\Phi(x,q)\big] \nonumber \\
			 &= u\big(t,\Phi(x,q)\big),
\end{align}
which establishes \eqref{eqn:koopman_bc}.
\end{proof}

The operator $\mathcal{A}_q$ corresponds to the standard infinitesimal generator of a diffusion process in $\Re^{n_q}$, capturing the local evolution of the observable due to the continuous stochastic dynamics within each mode.
Accordingly, the evolution of $u(t,x,q)$ is governed by a diffusion equation in the interior of each mode, with coupling across modes arising only through discrete transitions.

The compatibility condition \eqref{eqn:koopman_bc} enforces consistency of the observable across discrete transitions.
Unlike classical boundary conditions, this condition directly reflects the reset mechanism of the hybrid system by equating the observable before and after the jump.
The condition \eqref{eqn:koopman_bc} remains well-defined even when $\Phi(x,q)$ maps between spaces of different dimensions.
This highlights that the Koopman framework naturally accommodates hybrid systems with dimension-varying transitions, since observables are defined on the entire hybrid state space $\H$.

\section{Hybrid Fokker--Planck Equation with Dimension Change}

In this section, we derive the adjoint of the stochastic Koopman generator, which governs the evolution of probability densities.
While the Koopman operator describes the evolution of observables, its adjoint characterizes how probability mass evolves under the hybrid stochastic dynamics.
The key feature is that probability transfer across guards is described via the pushforward of the reset map.

\subsection{Hybrid Fokker--Planck Equation}

\begin{theorem}\label{thm:fp}
	Consider the stochastic hybrid system defined in Section II.
	Then, for each mode $q\in\Q$, the probability density $p(t,x,q)$ satisfies
	\begin{align}
		\left(\frac{\partial p(t,x,q)}{\partial t} + \nabla\cdot Y_q(t,x)\right)dx = d\eta_q(t), \label{eqn:fp_measure_form}
	\end{align}
	on $\X_q$, where the probability current is
	\begin{align}
		Y_q(t,x) = p(t,x,q)X_q(x) - \frac{\sigma_q^2}{2}\nabla p(t,x,q),
	\end{align}
	and $\eta_q(t)$ is a measure on $\X_q$ given by
	\begin{align}
		\eta_q(t) = \sum_{r:\,\Phi(\G_r,r)\subset \X_q\times\{q\}} \Phi_*((Y_r\cdot N_r)\,dS_r), \label{eqn:eta_q}
	\end{align}
	where $N_r(x)$ denotes the outward unit normal vector to the guard $\G_r \subset \X_r$, and $dS_r$ is the corresponding surface measure on $\G_r$.
	Here the summation is taken over all modes $r\in\Q$ for which the reset map $\Phi$ sends points on the guard $\G_r$ into mode $q$, i.e., $\Phi(\G_r,r)\subset \X_q\times\{q\}$.
	Also, $\Phi_*$ denotes the pushforward of the measure by the reset map~\cite{bogachev2007measure}.
\end{theorem}

\begin{proof}
	By the duality of the stochastic Koopman operator and the Frobenius--Perron operator, we have
	\begin{align}
		\pair{\frac{\partial p}{\partial t}, f} = \pair{p, \mathcal{A}f}, \label{eqn:tmp00}
	\end{align}
	for every test function $f:\H\to\Re$ in the domain of the Koopman generator.
	In particular, $f$ is smooth in the interior of each mode and satisfies the compatibility conditions on the boundary
	\begin{align}
		f(x,q)=f(\Phi(x,q)), \qquad x\in\G_q, \label{eqn:koopman_bc_f}\\
		\nabla f(x,q)\cdot N_q(x)=0, \qquad x\in\Gamma_q, \label{eqn:reflecting_bc_f}
	\end{align}
	where the first condition enforces consistency across the reset map and the second corresponds to reflecting boundary behavior.

	Using \eqref{eqn:koopman_generator}, we have
	\begin{align}
		\pair{p,\mathcal{A} f} = \sum_{q\in\Q}\int_{\X_q} \Bigl(X_q\cdot \nabla f + \frac{\sigma_q^2}{2}\Delta f\Bigr)p\,dx.
		\label{eqn:tmp0}
	\end{align}
	Applying integration by parts, via the divergence theorem, to the drift term yields
	\begin{align*}
		\int_{\X_q} pX_q\cdot \nabla f\,dx = \int_{\partial\X_q} fpX_q\cdot N_q\,dS_q - \int_{\X_q} f\nabla\cdot(pX_q)\,dx,
	\end{align*}
	and for the diffusion term,
	\begin{align*}
		\int_{\X_q} p\Delta fdx & = \int_{\partial\X_q} p\nabla f\cdot N_qdS_q - \int_{\partial\X_q} f\nabla p\cdot N_qdS_q\\
		& \quad + \int_{\X_q} f\,\Delta p\,dx,
	\end{align*}
	where $dS_q$ denotes the surface measure on $\partial \X_q$.
	Substituting these into \eqref{eqn:tmp0}, we obtain
	\begin{align}
		\pair{p,\mathcal{A} f} & = \sum_{q\in\Q}\int_{\X_q} -f\, \nabla\cdot Y_q\, dx + \sum_{q\in\Q}\int_{\partial\X_q} f\,(Y_q\cdot N_q)\,dS_q\nonumber\\
		& \quad + \sum_{q\in\Q}\int_{\partial\X_q} \frac{\sigma_q^2}{2}\,p\,\nabla f\cdot N_q\,dS_q.
		\label{eqn:tmp1}
	\end{align}

	Next, we simplify the last two terms, which are boundary integrals over $\partial \X_q$.
	Recall that the boundary is decomposed as $\partial \X_q=\G_q\cup\Gamma_q$.
	On the reflecting boundary $\Gamma_q$, the probability current satisfies $Y_q\cdot N_q=0$, and the test function satisfies $\nabla f\cdot N_q=0$.
	Thus, both boundary integrals in \eqref{eqn:tmp1} vanish on $\Gamma_q$, and only the guard contributions remain.
	In other words,
	\begin{align*}
		&\sum_{q\in\Q}\int_{\partial\X_q} f\,(Y_q\cdot N_q)\,dS_q + \sum_{q\in\Q}\int_{\partial\X_q} \frac{\sigma_q^2}{2}\,p\,\nabla f\cdot N_q\,dS_q\\
		&= \sum_{q\in\Q}\int_{\G_q} f\,(Y_q\cdot N_q)\,dS_q + \sum_{q\in\Q}\int_{\G_q} \frac{\sigma_q^2}{2}\,p\,\nabla f\cdot N_q\,dS_q.
	\end{align*}

	On the guard $\G_q$, the process is instantaneously reset and does not remain on the guard, so the interior density has zero trace on $\G_q$. Hence
	\begin{align*}
		p(t,x,q)\,\nabla f(x,q)\cdot N_q = 0, \qquad x\in\G_q.
	\end{align*}
	Therefore, the second boundary integral vanishes entirely.

	The remaining boundary contribution is
	\begin{align*}
		\sum_{q\in\Q}\int_{\G_q} f(x,q)\,(Y_q\cdot N_q)\,dS_q.
	\end{align*}
	Let $\G_r$ be a guard in mode $r$, and suppose that its reset image lies in mode $q$, i.e., $\Phi(\G_r,r)\subset \X_q\times\{q\}$.
	Using \eqref{eqn:koopman_bc_f},
	\begin{align*}
		\int_{\G_r} f(x,r)\,(Y_r\cdot N_r)\,dS_r = \int_{\G_r} f(\Phi(x,r))\,(Y_r\cdot N_r)\,dS_r.
	\end{align*}
	By the definition of the pushforward measure, this can be written as
	\begin{align*}
		&\int_{\G_r} f(\Phi(x,r))\,(Y_r\cdot N_r)\,dS_r\\
		& = \int_{\X_q} f(y,q)\,d\!\left[\Phi_*\bigl((Y_r\cdot N_r)\,dS_r\bigr)\right](y),
	\end{align*}
	which is the standard integral transformation formula for image measures \cite[Theorem 3.6.1]{bogachev2007measure}.
	From the definition \eqref{eqn:eta_q}, the total guard contribution is rearranged into
	\begin{align*}
		\sum_{r\in\Q}\int_{\G_r} f(x,r)\,(Y_r\cdot N_r)\,dS_r = \sum_{q\in\Q}\int_{\X_q} f(x,q)\,d\eta_q(t).
	\end{align*}
	Substituting this into \eqref{eqn:tmp1} and combining it with \eqref{eqn:tmp00}, we obtain
	\begin{align*}
		\pair{\frac{\partial p}{\partial t}, f} = \sum_{q\in\Q}\int_{\X_q} -f\, \nabla\cdot Y_q\, dx + \sum_{q\in\Q}\int_{\X_q} f\,d\eta_q(t).
	\end{align*}
	Since this holds for arbitrary test functions $f$, it follows that
	\begin{align}
		\left(\frac{\partial p(t,x,q)}{\partial t} + \nabla\cdot Y_q(t,x)\right)dx = d\eta_q(t)
	\end{align}
	as measures on $\X_q$, which shows \eqref{eqn:fp_measure_form}.
\end{proof}

The evolution equation \eqref{eqn:fp_measure_form} expresses conservation of probability within each mode, where the divergence term $\nabla\cdot Y_q$ accounts for the continuous stochastic flow, and the measure-valued term $\eta_q(t)$ represents probability injected into mode $q$ through reset events.
Specifically, $\eta_q(t)$ is the pushforward of the outgoing probability flux across guards in other modes, mapped into $\X_q$ by the reset map $\Phi$.
Thus, the hybrid Fokker--Planck equation consists of a classical diffusion--advection equation in the interior, coupled with a nonlocal source term induced by boundary-triggered jumps.

\subsection{Special Cases}
Several special cases provide additional insight.

\emph{(i) No jumps.}
If there are no guards, i.e., $\G_q=\emptyset$ for all $q$, then $\eta_q(t)=0$, and \eqref{eqn:fp_measure_form} reduces to the standard Fokker--Planck equation
\begin{align*}
	\frac{\partial p}{\partial t} + \nabla\cdot Y_q = 0,
\end{align*}
on each mode independently.

\emph{(ii) Purely deterministic flow.}
If $\sigma_q=0$, then $Y_q=pX_q$, and the equation reduces to a hybrid Liouville equation with measure-valued source:
\begin{align*}
	\left(\frac{\partial p}{\partial t} + \nabla\cdot(pX_q)\right)dx = d\eta_q(t),
\end{align*}
where probability is transported deterministically in the interior and redistributed through reset events.

\emph{(iii) No reset across modes.}
If $\Phi(\G_q,q)\subset \X_q\times\{q\}$ for all $q$, then the system does not change mode, and $\eta_q(t)$ represents redistribution within the same domain $\X_q$.
In this case, the equation describes a diffusion process with nonlocal boundary-induced source within a single mode.

\emph{(iv) Strong form.}
If $\eta_q(t)$ admits a density with respect to the Lebesgue measure, i.e., $d\eta_q(t)=s_q(t,x)\,dx$, then \eqref{eqn:fp_measure_form} can be written in strong form as
\begin{align*}
	\frac{\partial p}{\partial t} + \nabla\cdot Y_q = s_q(t,x),
\end{align*}
where $s_q(t,x)$ captures the redistributed probability mass from reset events.
In general, however, $\eta_q(t)$ is singular and supported on lower-dimensional subsets induced by the image of the guards.

\emph{(v) Flux balance across guards.}
The measure $\eta_q(t)$ ensures global conservation of probability across modes: the total outgoing flux through all guards equals the total incoming mass distributed by the pushforward measures.
This provides a consistent coupling between modes even when their dimensions differ.

\section{Example: Coalescing and Splitting Particles}

We present a stochastic hybrid system in which two particles evolve on a one-dimensional line segment and collapse into a single particle when they become sufficiently close, but may subsequently split back into two particles.
This example illustrates a hybrid system in which the dimension of the continuous state space changes across modes.

There are two modes $\Q=\{1,2\}$.
The first mode corresponds to the evolution of two particles, and the second mode describes the dynamics of the merged particle. 
The hybrid dynamics in each mode are described as follows.

\subsection{Mode 1: Two particles on $[0,1]^2$}

Let $x_A,x_B \in [0,1]$ denote the positions of two particles.
For $\epsilon\in(0,1)$, define the continuous state space
\begin{align}\label{eqn:domain_M1}
	\X_1 = \{(x_A,x_B)\in[0,1]^2 \mid |x_A - x_B| > \epsilon \}.
\end{align}
This corresponds to the unit square with a diagonal band of width $\epsilon$ removed, where the two particles are sufficiently close, as illustrated in \Cref{fig:guard1}.

In the interior of $\X_1$, the state evolves according to
\begin{align}\label{eqn:dx1}
	d x_1 = X_1(x_1)\,dt + \sigma_1 dW_1,
\end{align}
where $x_1=(x_A,x_B)$, and $X_1:\X_1\to\mathbb{R}^2$ is a vector field. 
Also, $\sigma_1>0$, and $W_1(t)\in\mathbb{R}^2$ is a Wiener process.

The outer boundary is reflecting, meaning that trajectories are reflected back into the interior upon reaching it. 
Specifically,
\begin{align}
	\Gamma_1 = \{(x_A,x_B)\in[0,1]^2 \mid x_A\in\{0,1\} \text{ or } x_B\in\{0,1\}\}.
\end{align}

The guard is the one-dimensional set where the distance between the two particles equals $\epsilon$:
\begin{align}\label{eqn:G1}
	\G_1 = \{(x_A,x_B)\in[0,1]^2 \mid |x_A - x_B| = \epsilon\}.
\end{align}
Thus, the boundary decomposes as $\partial\X_1 = \Gamma_1 \cup \G_1$.

When the guard is reached, the two particles collapse into a single particle located at their midpoint, and the system switches to Mode 2.
The reset map is given by
\begin{align}\label{eqn:Phi_1to2}
	\Phi((x_A,x_B),1) = \left(\frac{x_A + x_B}{2},\,2\right).
\end{align}

\subsection{Mode 2: Single particle on $[0,1]$}

Let $x_C\in[0,1]$ denote the position of the merged particle.
The continuous state space is
\begin{align}
	\X_2 = [0,1].
\end{align}

The dynamics are given by
\begin{align}\label{eqn:dx2}
	d x_2 = X_2(x_2)\,dt + \sigma_2 dW_2,
\end{align}
where $X_2:\X_2\to\mathbb{R}$ is a vector field, $W_2(t)\in\mathbb{R}$ is a Wiener process, and $\sigma_2>0$.

The boundary at $x=0$ is reflecting, and the boundary at $x=1$ is a guard. 
Thus,
\begin{align}
	\Gamma_2 = \{0\}, \quad \G_2 = \{1\}.
\end{align}
At the guard, the mass is split into two parts at a fixed location $x_1^* = (x^*_A,x^*_B)\in\X_1$ and the system switches to Mode~1. 
The reset map is
\begin{align}\label{eqn:Phi_2to1}
	\Phi(x_C=1, 2) = ((x^*_A, x^*_B), 1).
\end{align}

This system exhibits a change in dimension between $\dim(\X_1)=2$ to $\dim(\X_2)=1$ via the reset maps.
Since the reset image satisfies $\Phi(\G_1)\subset \mathrm{int}(\X_2)$,
probability flux leaving Mode 1 through $\G_1$ is injected into the interior of $\X_2$.
Consequently, the Fokker--Planck equation for Mode 2 contains a measure-valued source term induced by the pushforward of the guard flux.
Also, the Fokker--Planck equation for Mode 1 includes a source term given by a Dirac measure.

\begin{figure}
	\begin{center}
		\footnotesize\selectfont
		\begin{tikzpicture}[scale=5,>=Latex]
			\def\eps{0.18}

	% axes
			\draw[->] (-0.02,0) -- (1.10,0) node[below] {$x_A$};
			\draw[->] (0,-0.02) -- (0,1.10) node[left] {$x_B$};
			\node[below] at (1,0) {$1$};
			\node[left] at (0,1) {$1$};
			\node[below left] at (0,0) {$0$};

	% outer square
			\draw[thick] (0,0) rectangle (1,1);

	% --- Shade X_1 (outside the diagonal band) ---
	% Upper region: x_B >= x_A + eps
			\fill[gray!12]
				(0,\eps) -- (0,1) -- (1,1) -- (1-\eps,1) -- cycle;

	% Lower region: x_B <= x_A - eps
			\fill[gray!12]
				(0,0) -- (\eps,0) -- (1,1-\eps) -- (1,0) -- cycle;

	% Guard lines
			\draw[very thick] (\eps,0) -- (1,1-\eps)
				node[pos=0.6, below right, black, rotate=45] {$\G_1^+$};
			\draw[very thick] (0,\eps) -- (1-\eps,1)
				node[pos=0.4, above left, black, rotate=45] {$\G_1^-$};

	% Middle excluded band (visual hint)
			\draw[dashed] (0,0) -- (1,1);
			\node at (0.72,0.70) {\shortstack[c]{$|x_A-x_B|$\\$<\epsilon$}};

	% Label X_1
			\node at (0.2,0.85) {$\X_1$};

		\end{tikzpicture}
	\end{center}
	\caption{Illustration of the continuous state space $\X_1$ and the guard $\G_1$ for the first mode}
	\label{fig:guard1}
\end{figure}
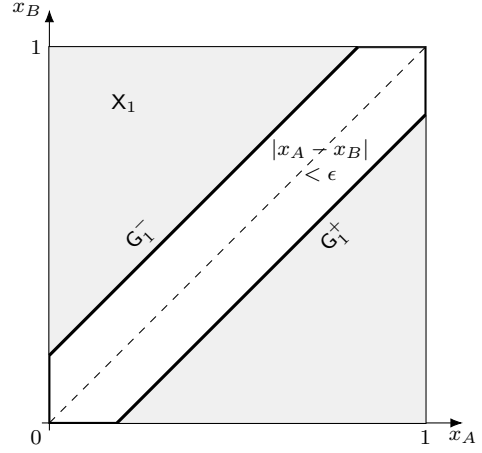

\subsection{Derivation of Source Measure $\eta_1$}

The source measure $\eta_1$ is generated by the outgoing probability flux from Mode~2 through the guard $\G_2=\{1\}$ and its reset into Mode~1.
By definition,
\begin{align}
	\eta_1(t)=\Phi_*\bigl((Y_2\cdot N_2)\,dS_2\bigr).
\end{align}
Since $\X_2=[0,1]$, the outward unit normal at the guard point $x_C=1$ is simply $N_2(1)=1$.
Moreover, because $\G_2=\{1\}$ is a zero-dimensional manifold, the induced surface measure is the counting measure on that point.
Hence
\begin{align}
	(Y_2\cdot N_2)\,dS_2 = Y_2(t,1)\,\delta_1(x_C),
\end{align}
where the probability current in Mode~2 is
\begin{align}
	Y_2(t,x_C)=X_2(x_C)\,p(t,x_C,2)-\frac{\sigma_2^2}{2}\frac{\partial p(t,x_C,2)}{\partial x_C}.
\end{align}

Since the reset map at the guard is constant,
\begin{align}
	\Phi(1,2)=((x^*_A,x^*_B),1),
\end{align}
the entire outgoing flux from $\G_2$ is mapped to the single point $(x^*_A,x^*_B)\in\X_1$.
Therefore, the pushforward measure is a Dirac mass supported at the reset location:
\begin{align}
	d\eta_1(t)=Y_2(t,1)\,\delta_{(x^*_A,x^*_B)}(x_A,x_B),
\end{align}
where $Y_2(t,1)$ denotes the scalar outward flux at the boundary, i.e., $(Y_2(t,1)\cdot N_2(1))$.

Thus, the source measure $\eta_1$ is singular with respect to the two-dimensional Lebesgue measure on $\X_1$.
This is a consequence of the fact that the reset map sends the zero-dimensional guard $\G_2$ to a single point in the two-dimensional state space $\X_1$.
Hence the reset from Mode~2 to Mode~1 appears in the Fokker--Planck equation for Mode~1 as a point source located at $(x^*_A,x^*_B)$.

\subsection{Derivation of Source Measure $\eta_2$}

The source measure $\eta_2$ is given by
\begin{align}
	\eta_2(t)=\Phi_*\bigl((Y_1\cdot N_1)\,dS_1\bigr).
\end{align}
The guard $\G_1$ consists of two branches.
Specifically, we have $\G_1 = \G_1^+\cup \G_1^-$ with
\begin{align*}
	\G_1^+ & = \{(x_A,x_B)\in[0,1]^2 \mid x_A-x_B=\epsilon\}, \\
	\G_1^- & = \{(x_A,x_B)\in[0,1]^2 \mid x_A-x_B=-\epsilon\}.
\end{align*}
Each branch of $\G_1$ is a straight line with slope $1$. 
Hence the outward unit normal vectors are obtained directly from the geometry as
\begin{align}
	N_1^+ = \frac{1}{\sqrt2}\begin{bmatrix}-1\\ 1\end{bmatrix} \quad \text{on } \G_1^+, 
\end{align}
and $N_1^- = - N_1^+$ on $\G_1^-$.

Next, since $\eta_2$ is defined as the pushforward of a surface measure supported on $\G_1$, we parameterize $\G_1$.
Specifically, we choose the midpoint as the parameter:
\begin{align}
	x_C = \frac{x_A+x_B}{2}.
\end{align}
Then any point $(x_A,x_B)\in\G_1$ on the guard is written as
\begin{align}
	(x_A,x_B)=\begin{cases}
		\left(x_C+\frac{\epsilon}{2},\,x_C-\frac{\epsilon}{2}\right)&  \text{ on $\G_1^+$},\\
		\left(x_C-\frac{\epsilon}{2},\,x_C+\frac{\epsilon}{2}\right)&  \text{ on $\G_1^-$}.
	\end{cases}
\end{align}
Since $x_A,x_B\in[0,1]$ and $x_A = x_C \pm \frac{\epsilon}{2}$, $x_B = x_C \mp \frac{\epsilon}{2}$ on the guard, it follows that
\begin{equation}
	\frac{\epsilon}{2} \le x_C \le 1 - \frac{\epsilon}{2},   \label{eqn:xC_range}
\end{equation}
which defines the admissible range of the midpoint coordinate.
Since each branch is a line of slope \(1\), the surface measure on the line is
\begin{align}
	dS_1=\sqrt2\,dx_C.
\end{align}

Next, the probability current at Mode 1 is
\begin{align}
	Y_1(t,x_A,x_B) &= X_1(x_A,x_B)\,p(t,x_A,x_B,1) \nonumber \\
				   & \quad -\frac{\sigma_1^2}{2}\nabla p(t,x_A,x_B,1)\nonumber \\
				   &=
				   \begin{bmatrix}
					   Y_A(t,x_A,x_B)\\
					   Y_B(t,x_A,x_B)
				   \end{bmatrix},\label{eqn:mass_Y1}
\end{align}
where
\begin{align}
	Y_A &= X_{1_A}\,p_1 -\frac{\sigma_1^2}{2}\frac{\partial p_1}{\partial x_A},\\
	Y_B&= X_{1_B}\,p_1 -\frac{\sigma_1^2}{2}\frac{\partial p_1}{\partial x_B}.
\end{align}
with $X_1 = (X_{1_A}, X_{1_B})\in\Re^2$.

Then on $\G_1^+$, we have
\begin{align}
	(Y_1\cdot N_1^+)\,dS_1 &= \frac{1}{\sqrt2}(-Y_A+Y_B)\,\sqrt2\,dx_C \nonumber \\
						   & = (-Y_A+Y_B)\,dx_C,
\end{align}
and similarly on $\G_1^-$,
\begin{align}
	(Y_1\cdot N_1^-)\,dS_1 = (Y_A-Y_B)\,dx_C.
\end{align}

The reset map is written as
\begin{align}
	\Phi((x_A,x_B),1)=\left(\frac{x_A+x_B}{2},\,2\right)=(x_C,2).
\end{align}
Since $\Phi$ maps each branch of the guard diffeomorphically onto an interval in $\X_2$, the pushforward measure $\eta_2$ is absolutely continuous with respect to the Lebesgue measure.
Thus, the source measure can be written as
\begin{align}
	d\eta_2(t)=s_2(t,x_C)\,dx_C,
\end{align}
where the source density is obtained by summing the contributions from the two branches:
\begin{align}
	s_2(t,x_C) &= -Y_A\left(t,x_C+\frac{\epsilon}{2},x_C-\frac{\epsilon}{2}\right) \nonumber\\
			   & \quad + Y_B\left(t,x_C+\frac{\epsilon}{2},x_C-\frac{\epsilon}{2}\right) \nonumber\\
			   &\quad+ Y_A\left(t,x_C-\frac{\epsilon}{2},x_C+\frac{\epsilon}{2}\right) \nonumber\\ 
			   &\quad - Y_B\left(t,x_C-\frac{\epsilon}{2},x_C+\frac{\epsilon}{2}\right),
			   \label{eqn:s2_line}
\end{align}
for $x_C$ satisfying \eqref{eqn:xC_range}.

This expression shows that the source term is determined by the imbalance of probability flux between the two particles at the moment of collapse.
% In this example, the reset map \(\Phi:\G_1\to\X_2\) is a local diffeomorphism between manifolds of equal dimension, i.e., $\mathrm{dim}(\G_1)=\mathrm{dim}(\X_2)$. 
% As a result, the pushforward of the boundary flux reduces to a pointwise evaluation with a Jacobian factor, and no integration over the guard is required.

\subsection{Fokker--Planck Equation}
Consequently, the Fokker--Planck equation is
\begin{align}
	\frac{\partial p(t,x_A,x_B,1)}{\partial t} & = - \frac{\partial}{\partial x_A} Y_A(t,x_A,x_B)\nonumber\\
											   & \quad - \frac{\partial}{\partial x_B} Y_B(t,x_A,x_B)\nonumber\\
											   & \quad + Y_2(t,1)\,\delta_{(x^*_A,x^*_B)}(x_A,x_B), \label{eqn:FP_mass_1}\\
	\frac{\partial p(t,x_C,2)}{\partial t} &= -\frac{\partial}{\partial x_C}Y_2(t,x_C) + s_2(t,x_C),\label{eqn:FP_mass_2}
\end{align}
with the boundary conditions given by
\begin{gather}
	p_1(t,x_A,x_B) = 0 \text{ at $\G_1$},\\
	Y_1(t,x_A,x_B) \cdot N_1 = 0 \text{ at $\Gamma_1$},\\
	p_2(t,x_C) = 0 \text{ at $\G_2$},\\
	Y_2(t, x_C) = 0 \text{ at $\Gamma_2$}.
\end{gather}
Hence Mode~1 follows the Fokker--Planck equation on \(\X_1\) with reflecting outer boundary and vanishing trace on the guard, together with the reset-induced Dirac source generated by outgoing probability flux from Mode~2.
Similarly, Mode~2 follows the Fokker--Planck equation on \(\X_2\) with reflecting boundary at \(x_C=0\), vanishing trace at the guard \(x_C=1\), and the source term \(s_2\) induced by outgoing probability flux from Mode~1.

This example highlights two fundamentally different types of reset-induced source terms.
When probability flows from Mode 1 to Mode 2, the pushforward of boundary flux produces an absolutely continuous source density in the lower-dimensional state space.
In contrast, when probability flows from Mode 2 back to Mode 1, the pushforward of flux from a lower-dimensional guard results in a singular source concentrated on a curve in the higher-dimensional space.

This demonstrates that dimension-changing resets naturally generate both smooth and singular contributions, which cannot be captured within classical boundary-condition-based formulations.
More broadly, it illustrates how the theorem accommodates dimension-changing resets by converting boundary flux in one mode into an interior source term in another.

\subsection{Numerical Example}

% Requires:
% \usepackage{tikz}
% \usetikzlibrary{arrows.meta,calc}

\begin{figure}
	\begin{center}
		\footnotesize\selectfont
		\begin{tikzpicture}[scale=4, >=Latex]
			\def\eps{0.18}
			\def\alphaone{2.5}
			\def\gammaone{0.01}
			\def\vecscale{0.1}   % scale factor for arrows

			% axes
			\draw[->] (-0.03,0) -- (1.1,0) node[below] {$x_A$};
			\draw[->] (0,-0.03) -- (0,1.1) node[left] {$x_B$};

			% outer box
			\draw[thick] (0,0) rectangle (1,1);

			% shade X_1
			\fill[gray!12] (0,\eps) -- (0,1) -- (1,1) -- (1-\eps,1) -- cycle;
			\fill[gray!12] (0,0) -- (\eps,0) -- (1,1-\eps) -- (1,0) -- cycle;

			% guard lines
			\draw[very thick] (\eps,0) -- (1,1-\eps);
			\draw[very thick] (0,\eps) -- (1-\eps,1);

			% optional dashed diagonal
			\draw[dashed] (0,0) -- (1,1);

			% vector field samples
			% choose grid points manually
			\foreach \xa in {0.08,0.18,0.28,0.38,0.48,0.58,0.68,0.78,0.88}{
				\foreach \xb in {0.08,0.18,0.28,0.38,0.48,0.58,0.68,0.78,0.88}{

					% only draw on X_1, i.e. |x_A-x_B| > eps
					\pgfmathsetmacro{\dist}{abs(\xa-\xb)}

					\ifdim \dist pt>\eps pt
						% vector field
						\pgfmathsetmacro{\u}{\xa*(1-\xa)*(\alphaone*(\xb-0.5)-\gammaone*(\xa-0.5))}
						\pgfmathsetmacro{\v}{\xb*(1-\xb)*(-\alphaone*(\xa-0.5)-\gammaone*(\xb-0.5))}

						% normalize slightly for visualization
						\pgfmathsetmacro{\normuv}{veclen(\u,\v)}
						\ifdim \normuv pt>0pt
							\pgfmathsetmacro{\du}{\vecscale*\u/(\normuv+0.00001)}
							\pgfmathsetmacro{\dv}{\vecscale*\v/(\normuv+0.00001)}

							\draw[->,blue!70!black] (\xa,\xb) -- ++(\du,\dv);
						\fi
					\fi
				}
			}
		\end{tikzpicture}
	\end{center}
	\vspace*{-0.5cm}
	\caption{Vector field on $\X_1$}\label{fig:X1}
\end{figure}
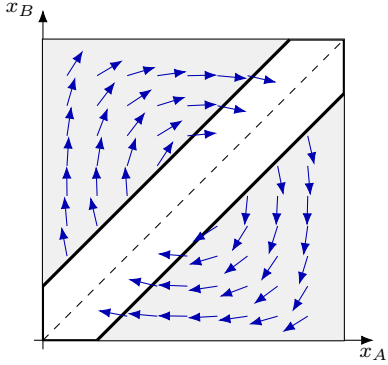

For $\alpha_1, \gamma_1>0$, the vector field for Mode~1 is chosen as
\begin{align}\label{eqn:ex_X1}
	X_1 = \begin{bmatrix} x_A(1-x_A) (\alpha_1(x_B-0.5)-\gamma_1(x_A-0.5)) \\
		x_B(1-x_B) (-\alpha_1(x_A-0.5)-\gamma_1(x_B-0.5)) \\
	\end{bmatrix},
\end{align}
where the first factor ensures that the vector field is parallel with the outer boundary at $\Gamma_1$, and the second factor creates clockwise rotation about the center $(0.5,0.5)$, which is illustrated in \Cref{fig:X1}.
The last factor yields the convergence to the center.
For Mode~2,
\begin{align}\label{eqn:ex_X2}
	X_2 = -\gamma_2 (x_C-2),
\end{align}
which corresponds to the shift to the right in $\X_2=[0,1]$ for $\gamma_2>0$.
The drift parameters are chosen as $\alpha_1 = 2.5, \gamma_1 = 0.2, \gamma_2 = 0.2$.
Next, the diffusion coefficients are $ \sigma_1 = 0.01, \sigma_2 = 0.01 $.

Regarding the jump characteristics, the band width of the guard $\G_1$ is chosen as $\epsilon = 0.05$, while the flux from $\G_2$ is mapped to the source $(x^*_A,x^*_B) = (0.5, 0.8) \in \X_1$.
In the numerical scheme, the reset from Mode 2 to 1 is implemented as a localized source using a smoothed approximation of the Dirac delta distribution to ensure numerical stability while preserving total mass.

The initial density is supported entirely in Mode 1 and is given by the Gaussian distribution $p(0, x_A, x_B, 1) = \mathcal{N}(\mu_{AB}, \sigma_{AB}^2) $ where $\mu_{AB} = [\mu_A, \mu_B] = [0.6, 0.2] $ and $\sigma_{AB}^2 = \text{diag}(0.008, 0.004)$, restricted to the admissible region ${|x_A - x_B| > \epsilon}$.
The density is normalized so that the total probability mass is unity.
In Mode 2, the initial density is $p(0, x_C, 2) = 0$.

The hybrid Fokker–Planck equations are discretized using a finite volume method on structured grids with resolutions $N_{x_A} = 201$, $N_{x_B} = 201$, $N_{x_C} = 201$.
A second-order MUSCL (Monotonic Upstream-centered Scheme for Conservation Laws)-type reconstruction with a minmod slope limiter is used for the advective fluxes~\cite{leveque2002finite}, while diffusion is discretized using centered differences. 
Time integration is performed using a strong stability preserving second-order Runge–Kutta (SSP-RK2) scheme with $N_t = 20000$ time steps up to a final time $T = 5$.
The scheme directly discretizes the probability current rather than the density alone, ensuring that inter-mode transfer is governed by flux consistency rather than ad hoc source approximations.

The corresponding evolution of the densities is illustrated in \Cref{fig:density_snap}.
Starting from the initial Gaussian distribution in Mode 1, the rotational drift field in \eqref{eqn:ex_X1} drives the density $p_1$ toward one branch of the guard $\G_1^+$ (\Cref{fig:density_snap}.(a)--(b)). 
As the trajectories approach the region where $|x_A - x_B| \le \epsilon$, probability mass transitions to Mode 2 through the reset map \eqref{eqn:Phi_1to2} (\Cref{fig:density_snap}.(b)). 
In Mode 2, the drift \eqref{eqn:ex_X2} transports the density $p_2$ toward the boundary $x_C = 1$ on the right, where it is reset to the point $x^*=(x_A^*,x_B^*)$ in Mode 1 (\Cref{fig:density_snap}.(c)--(d)). 
The subsequent evolution in Mode 1 causes the density $p_1$ to approach the opposite branch $\G_1^-$, leading to repeated transitions between the two modes (\Cref{fig:density_snap}.(e)). 
This interplay between continuous evolution and discrete transitions results in a cyclic hybrid behavior. 

%Special care is taken in the treatment of the hybrid interface. 
%Fluxes across the guard $\G_1$ are consistently transferred to Mode 2 through a conservative projection onto the $x_C$ grid. 
%Likewise, the reset from Mode 2 to Mode 1 is implemented via a normalized kernel to ensure that the total mass injected into Mode 1 matches the outgoing flux from Mode 2.

To verify the numerical results, the evolution of the probability mass in each mode, as well as the total mass, is computed.
The results in \Cref{fig:total_mass} confirm that the total mass is conserved throughout the simulation, demonstrating the conservative nature of the proposed scheme.
In addition, a Monte Carlo simulation is performed for further validation.
Specifically, $2\times 10^5$ particles are initialized according to the same truncated Gaussian distribution in Mode~1 and propagated using the Euler–Maruyama scheme.
The density approximated using the Monte Carlo ensemble in \Cref{fig:density_mc_snap} follows the hybrid Fokker--Planck results, with the only difference being the smoothed approximation of the Dirac distribution in the source term of \eqref{eqn:FP_mass_1}.
Furthermore, as shown in \Cref{fig:total_mass}, the mass evolution predicted by the hybrid Fokker–Planck equation closely matches that obtained from the Monte Carlo simulation.

%Rather than comparing instantaneous jump rates, which are sensitive to time discretization, we compare the cumulative flux of particles transitioning from Mode 2 to Mode 1 over time. 
%The cumulative flux obtained from the Monte Carlo simulation closely matches that computed from the numerical solution of the Fokker–Planck equations, providing strong evidence of consistency between the two approaches.

\begin{figure}
	\centering
	\foreach \i / \t in {00/0}{
		\subcaptionbox{$t =\t $}
		{\scriptsize\begin{tikzpicture}
			\node {\includegraphics[width=0.975\columnwidth]{./Figs/combined_0\i000.pdf}};		
			\node[white, left] at (-1.8,1.0) {$x^*$};
		\end{tikzpicture}\vspace*{-0.45cm} }
	}
	\foreach \i / \t in {06/1.5, 12/3, 16/4, 20/5}{
		\subcaptionbox{$t =\t $}
		{\scriptsize\begin{tikzpicture}
			\node {\includegraphics[width=0.975\columnwidth]{./Figs/combined_0\i000.pdf}};		
			\node[white, left] at (-1.8,1.3) {$x^*$};
		\end{tikzpicture}\vspace*{-0.45cm} }
	}
	\caption{Probability density evolution for Mode 1 (left) and Mode 2 (right)}
	\label{fig:density_snap}
\end{figure}

\begin{figure}
	\centering
	\foreach \i / \t in {00/0}{
		\subcaptionbox{$t =\t $}
		{\scriptsize\begin{tikzpicture}
				\node {\includegraphics[width=0.975\columnwidth]{./Figs/mc_0\i000.pdf}};		
				\node[white, left] at (-1.8,1.0) {$x^*$};
			\end{tikzpicture}\vspace*{-0.45cm} }
	}
	\foreach \i / \t in {06/1.5, 12/3, 16/4, 20/5}{
		\subcaptionbox{$t =\t $}
		{\scriptsize\begin{tikzpicture}
				\node {\includegraphics[width=0.975\columnwidth]{./Figs/mc_0\i000.pdf}};		
				\node[white, left] at (-1.8,1.3) {$x^*$};
			\end{tikzpicture}\vspace*{-0.45cm} }
	}
	\caption{Monte Carlo density evolution for Mode 1 (left) and Mode 2 (right); compare with \Cref{fig:density_snap}}
	\label{fig:density_mc_snap}
\end{figure}

\begin{figure}[h]
	\centering
	\includegraphics[width=\columnwidth]{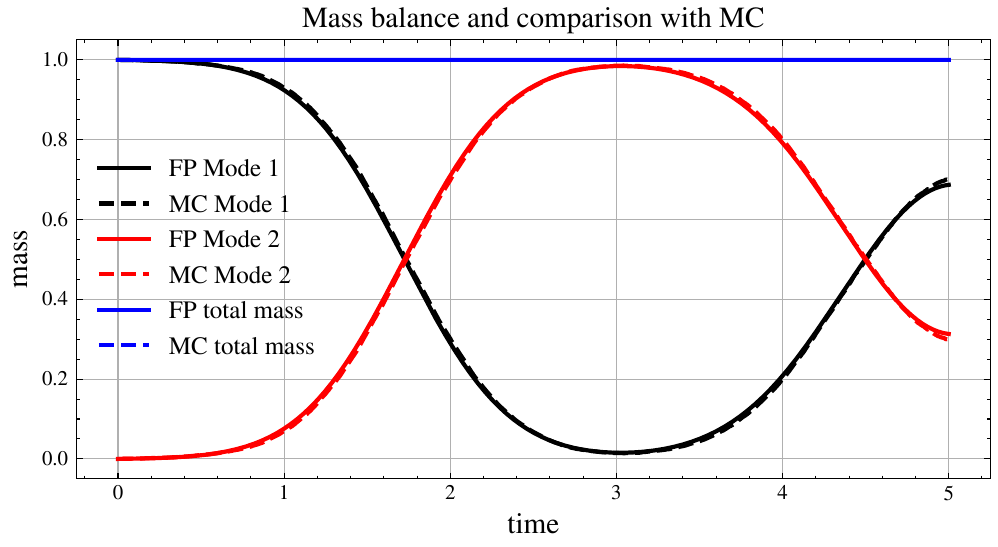}
	\caption{Probability mass evolution (solid: proposed hybrid Fokker-Planck equation, dashed: Monte-Carlo simulation)}
	\label{fig:total_mass}
\end{figure}

\section{Conclusions}

This paper presented a unified formulation of the hybrid Fokker--Planck equation for stochastic hybrid systems with reset maps that may change the dimension of the continuous state across modes.
By adopting a flux-driven, measure-theoretic perspective, probability transfer between modes is characterized through the mapping of boundary flux under the reset mechanism.
This provides a consistent description of uncertainty propagation that accommodates both smooth density contributions and singular source terms arising from dimension-changing transitions.

The proposed framework establishes a foundation for analyzing and computing probability evolution in stochastic hybrid systems with complex reset structures.
Future work includes extending the approach to higher-dimensional systems, incorporating control inputs, and developing efficient numerical methods for large-scale problems.

\bibliography{ref}
\bibliographystyle{IEEEtran}

\end{document}